\documentclass [a4paper,11pt]{amsart}

\usepackage[utf8]{inputenc}
\usepackage{xcolor}
\usepackage[pdftex]{graphicx} 
\usepackage{hyperref}
\DeclareGraphicsExtensions{.png,.pdf,.jpg}
\usepackage{amsmath}
\usepackage{amssymb}
\usepackage{latexsym}
\usepackage[all]{xy}
\usepackage{array}
\usepackage{bbm}
\usepackage{float}

\def\dim{\operatorname{dim}}

\def\Im{\operatorname{Im}}

\newcommand{\C}{\mathbb{C}}
\newcommand{\PP}{\mathbb{P}}

\newcommand{\A}{\mathcal{A}}
\newcommand{\Z}{\mathbb{Z}}

\newtheorem{teo}{Theorem}[section]

\newtheorem{propo}[teo]{Proposition}

\theoremstyle{definition}
\newtheorem{defi}[teo]{Definition}
\newtheorem{conj}[teo]{Conjecture}

\theoremstyle{definition}

\theoremstyle{definition}
\newtheorem{re}[teo]{Remark}

\begin{document}

\title[Image Milnor number formulas]{Image Milnor number formulas for weighted-homogeneous map-germs}
\author{Irma Pallar\'es Torres and Guillermo Pe\~nafort Sanchis }

\address{Basque Center for Applied Mathematics,
Alameda de Mazarredo 14, 48009 Bilbao, Bizkaia, Spain.}

\email{ipallares@bcamath.org}

\email{guillermo.penafort@uv.es}

\date{}

\thanks{The authors were partially supported by the ERCEA 615655 NMST 
Consolidator Grant and 
by the Basque Government through the BERC 2014-2017 program and by Spanish 
Ministry of Economy and
Competitiveness MINECO: BCAM Severo Ochoa excellence accreditation 
SEV-2013-0323.}

\subjclass[2010]{Primary 58K60; Secondary  58K40, 14C17} 
\keywords{Singularities of mappings, vanishing homology, weighted homogeneous singularities, characteristic classes, Thom polynomials}

\begin{abstract}
We give formulas for the image Milnor number of a weighted-homogeneous map-germ $(\C^n,0)\to(\C^{n+1},0)$, for $n=4$ and $5$, in terms of weights and degrees. Our expressions are obtained by a purely interpolative method, applied to a result by Ohmoto. We use our approach to recover the formulas for $n=2$ and $3$ due to Mond and Ohmoto, respectively. For $n\geq 6$, the method is valid as long as certain multi-singularity conjecture holds. 
\end{abstract}

\maketitle
\section{Introduction}
The study of $\mathcal A$-finite singular map-germs $F\colon (\C^n,0)\to(\C^{n+1},0)$ is rich in invariants defined by passing to a stable object. This theory deals with $\mathcal A$-classes, that is, map-germs up to coordinate changes in source and target. A germ is \emph{stable} if its $\mathcal A$-class does not change after a small perturbation, and it is \emph{$\mathcal A$-finite} if  stability fails at most on an isolated point. Some common invariants associated to $\mathcal A$-finite germs are the \emph{$0$-stable invariants}, the \emph{$\mathcal A_e$-codimension} and the \emph{image Milnor number} $\mu_I$. The $0$-stable invariants count the number of appearances of particular stable singularity types, the $\mathcal A_e$-codimension measures the number of parameters of a mini-versal unfolding and  $\mu_I$ counts the rank in the middle dimension of the homology of the image of a stable perturbation. 

All these invariants are hard to compute directly from the definition, but many of them can be computed as dimensions of suitable vector spaces. This accounts for many of the $0$-stable invariants and for the $\mathcal A_e$-codimension (see Section \ref{Afinite}), but not so far for $\mu_I$. This paper is devoted to establishing formulas of a different nature for $\mu_I$ in the case of weighted-homogeneous map-germs, extending results of Ohmoto and Mond. Apart from being interesting on their own, the formulas for $\mu_I$ bring us closer to the proof of Mond's conjecture, (see Conjecture \ref{MondConjecture}), which claims that $\mathcal A_e\textrm{-codim}\leq \mu_I$, with equality in the case of weighted-homogeneous map-germs. Thanks to results from \cite[Section 7]{Fernandez-de-Bobadilla:2016}, it suffices to check the statement of the conjecture for a family of finitely determined map-germs with unbound multiplicity. Given such a family, one can compute its $\mathcal A_e$-codimension via the already known formula, the only part missing is the $\mu_I$ computation. It is worth mentioning that the possibility of finding a formula that computes $\mu_I$ as the dimension of a vector space prior to proving Mond's conjecture is unlikely. This is because there is already a candidate for such a formula found in \cite{Fernandez-de-Bobadilla:2016} but, as explained there, proving that it actually computes $\mu_I$ is equivalent to proving Mond's conjecture.  The formulas found here are based on an  alternative, more topological approach that goes back to Thom \cite{Thom:1956} and connects the geometry of singular maps to certain characteristic classes. 

T. Ohmoto \cite{Ohmoto:2014} has adapted these techniques to show the existence of formulas computing the $0$-stable invariants and the image Milnor number, for weighted-homogenenous map-germs,  for $n\leq 5$, in terms of weights and degrees. The formulas are conjectured to exist for arbitrary $n$, see \cite{Kazarian2003Multising,Kazarian:2006,Ohmoto:2014}. While
the expressions for $0$-stable invariants follow easily from their \emph{Thom polynomials}, the image Milnor number formulas are harder to obtain.

The $\mu_I$ formulas predicted by Ohmoto \cite{Ohmoto:2014} have a specific form; they are rational functions with known denominator, whose numerator is obtained from the $n$-th degree truncation of the Segre-MacPherson Thom polynomial $tp^{\text{SM}}(\alpha_{\text{image}})$ series. This series is an extension of the classical Thom polynomial  $tp(\eta)$ of a stable singularity type $\eta$.  R. Rim\'any's \emph{restriction method} \cite{Rimanyi:2001}  allows to determine $tp(\eta)$  from calculations involving stable types of different dimensions. The method was adapted by Ohmoto \cite{Ohmoto:2014} to determine $tp^{\text{SM}}(\alpha_{\text{image}})$ up to degree three, giving $\mu_I$ formulas for $n=2$ and $3$ (the formula for $n=2$ was obtained previously by Mond  by different methods \cite{Mond:1991quasihomogeneous}).

Recent developments have made possible to compute truncations of the series $tp^{\text{SM}}(\alpha_{\text{image}})$ in a different manner: the series $tp^{\text{SM}}(\alpha_{\text{image}})$ has coefficients $b_\alpha\in \mathbb Q$ and variables $s_0$ and $c_i$. If $F\colon (\C^n,0)\to (\C^{n+1},0)$ is a weighted-homogeneous mapping with grading  $(w,d)$, then the image Milnor number $\mu_I(F)$ depends on the evaluation on certain functions $s_0(w,d),c_i(w,d)$ of the $n$-th truncation of $tp^{\text{SM}}(\alpha_{\text{image}})$.  Our goal is to find the $b_\alpha$ up to  degree $n$. For fixed $F$, there is a way to compute $\mu_I(F)$  with the software {\sc Singular}, based on results in  \cite{Fernandez-de-Bobadilla:2016} (see also \cite{Sharland2017ExamplesCorank3}). Having the value of $\mu_I(F)$ and the grading $(w,d)$ of  $F$ at hand, one can determine some relations between the coefficients  $b_\alpha$. Sampling enough map-germs $F$, one can determine the desired $b_\alpha$. Observe that this interpolative method does not involve the characteristic classes construction of Ohmoto's approach. Details are found in Section \ref{secOhmotoFormula}.
We use this method to recover the formulas known already and to derive new ones for all cases where the theory is stablished rigorously, that is, for $n=4$ and $5$.

The first steps of the process are easy, consisting only on sampling singularities found in the literature. Indeed, a surprisingly big portion of the interpolation can be completed just by sampling different gradings of stable germs. The challenge starts once the information coming from known singularities has been exhausted. On one hand, map-germs which are too simple do not yield new information about the $b_{\alpha}$ (for example, the case of $n=4$ requires sampling at least one map-germ with quintuple points, while the case of $n=5$ requires considering no less than three map-germs of corank two). On the other hand, degenerate candidates can be too complicated to compute their $\mu_I$, or to check  $\mathcal A$-finiteness. The difficulty of this work strives on navigating between these two extremes. In a series of remarks, we emphasize on the key strategies that have made our interpolative approach successful.

\section{Formulas for $\mu_I$ and $\# \eta$}\label{RESULTS}

In this section we list the formulas for the image Milnor number $\mu_I$ and $0$-stable invariants $\#\eta(F)$ of an $\mathcal{A}$-finite weighted-homogeneous map-germ \[F\colon(\C^n,0)\rightarrow (\C^{n+1},0),\]  for $n=4$ and $5$. The expressions depend on some coefficients $c_{k,n}$ and $s_{0,n}$ of Chern and Landweber-Novikov classes associated to $F$ (and not on the classes themselves). For the shake of completeness, these classes are introduced briefly in Section \ref{secOhmotoFormula}, the proofs being found in subsequent sections.

\subsection*{Image Milnor number for $\boldsymbol{n=4}$ and $\boldsymbol 5$}
Analogously to the classical hyper-surface setting of Milnor \cite{Milnor1968Singular-points}, the homotopy of the image of stable perturbation of a map-germ gives a well defined invariant. 
Let $F\colon(\C^n,0)\rightarrow (\C^{n+1},0)$ be an $\mathcal{A}$-finite map-germ and $f_y$ a stable perturbation of $F$. By \cite{Mond:1991}, the disentanglement $X_y$ of $F$ has the homotopy type of a wedge of $n$-spheres, and the number $\mu_I(F)$ of such spheres is the \emph{image Milnor number} of $F$. 

As mentioned in the Introduction, Ohmoto has shown that, for weighted-homogeneous map-germs, $\mu_I$ can be expressed in terms of the weights and degrees for $n\leq 5$. The restriction $n\leq 5$ comes from the fact that certain results are only known for Morin singularities. The expression for $n=2$ was obtained by Mond \cite{Mond:1991quasihomogeneous} with a different approach. Ohmoto in \cite{Ohmoto:2014} recovers the formula for $n=2$ and obtains the one for $n=3$. The following theorem, whose proof will be given in Section \ref{SecNew}, includes the two remaining cases.

\begin{teo} \label{FormulaNew}
Let $F\colon (\C^n,0)\to(\C^{n+1},0)$ be a weighted-homogeneous $\mathcal{A}$-finite map-germ with weights $w=(w_1,\ldots,w_n)$ and degrees $d=(d_0,\ldots, d_{n})$.
 If $n=4$,  then 
\begin{enumerate}
\item[]$\mu_I(F)=\dfrac{1}{\sigma_4}\big(\dfrac{1}{2!}(-s_0+c_1) \sigma_3+\dfrac{1}{3!}(s_0^2-c_1^2-c_2)\sigma_2\vspace{1.5mm}\\
\phantom{\mu_I(f)=}+\dfrac{1}{4!}(-s_0^3-2 s_0^2c_1+ s_0 c_1^2+16 s_0 c_2+2c_1^3-10c_1c_2)\sigma_1 \vspace{1.5mm} \\
\phantom{\mu_I(f)=}+\dfrac{1}{5!}(s_0^4+5s_0^3c_1+5s_0^2c_1^2-50s_0^2c_2-5s_0c_1^3-20s_0c_1c_2\vspace{1.5mm}\\
\phantom{\mu_I(f)=}+60s_0c_3-6c_1^4+34c_1^2c_2-64c_1c_3+108c_2^2+4c_4)\big).$
\end{enumerate}
If $n=5$, then
\begin{enumerate}
\item[]$\mu_I(F)=-\dfrac{1}{\sigma_5}\big(\dfrac{1}{2!}(-s_0+c_1) \sigma_4+\dfrac{1}{3!}(s_0^2-c_1^2-c_2)\sigma_3+\vspace{1.5mm} \\
\phantom{\mu_I(f)=}+\dfrac{1}{4!}(-s_0^3-2 s_0^2c_1+ s_0 c_1^2+16 s_0 c_2+2c_1^3-10c_1c_2)\sigma_2 \vspace{1.5mm}\\
\phantom{\mu_I(f)=}+\dfrac{1}{5!}(s_0^4+5s_0^3c_1+5s_0^2c_1^2-50s_0^2c_2-5s_0c_1^3-20s_0c_1c_2\vspace{1.5mm}\\
\phantom{\mu_I(f)=}+60s_0c_3-6c_1^4+34c_1^2c_2-64c_1c_3+108c_2^2+4c_4)\sigma_1\vspace{1.5mm}\\
\phantom{\mu_I(f)=}+\dfrac{1}{6!}(-s_0^5-9s_0^4c_1-25s_0^3c_1^2+110s_0^3c_2-15s_0^2c_1^3+270s_0^2c_1c_2\vspace{1.5mm}\\\phantom{\mu_I(f)=}-240s_0^2c_3+26s_0c_1^4+16s_0c_1^2c_2+24s_0c_1c_3-1138s_0c_2^2+336s_0c_4
\vspace{1.5mm}\\
\phantom{\mu_I(f)=}+24c_1^5-156c_1^3c_2+276c_1^2c_3+108c_1c_2^2-396c_1c_4+600c_2c_3)\big).$
\end{enumerate}
\end{teo}
 
 The coefficients $\sigma_k$, $c_k$ and $s_0$ are determined by $w$ and $d$ as follows: For fixed $n$, set 
$$\sigma_k=\sigma_{k,n}=\sum_{1\leq j_1<\ldots <j_k\leq n} w_{j_1}\cdot\ldots \cdot w_{j_k},$$
for $k=1,\dots, n$. To obtain the $c_{k}=c_{k,n}$, we set
$$\delta_k=\delta_{k,n}=\sum_{0\leq i_1<\ldots <i_k\leq {n}} d_{i_1}\cdot\ldots \cdot d_{i_k},$$ for $k=1,\dots, n+1$.
 Then,
$$c_{k,n}=\sum_{0\leq i\leq k}(-1)^{k-i}\delta_{i}\sum_{|\alpha|=k-i}w^{\alpha},$$
with the usual multi-index notation for $\alpha$.
Finally, $s_0=s_{0,n}$ is the rational function
$$s_0=\frac{\delta_{n+1}}{\sigma_{n}}.$$

\subsection*{Zero-stable invariants}
For any fixed $n$, certain stable multi-germs types appear, at most, on isolated points in the target of the stable multi-germs $F\colon(\C^n,S)\to (\C^{n+1},0)$. Such a stable type $\eta$ is called a \emph{$0$-stable type} for the dimensions $(n,n+1)$. Whenever a map-germ $F$ is stabilised, the target of the stable perturbation exhibits a certain number of multi-germs of type $\eta$. If $F$ is $\mathcal A$-finite, this number is independent of the chosen stabilisation and it is $\mathcal A$-invariant. This number is called the \emph{$0$-stable invariant} $\#\eta(F)$. We write $\# \eta$ for $\#\eta(F)$ if there is no risk of confusion.
 
As in the case of the image Milnor number, Ohmoto shows the existence of expressions for the $0$-stable invariants of weighted-homogeneous map-germs in terms of $\sigma_n, s_0$ and $c_k$, for $n\leq 5$. By \cite[Theorem 5.3]{Ohmoto:2014}, the $0$-invariants admit the expression
 \[\#\eta(F)=\frac{[tp(\eta)]_n}{\deg_1(\eta)w_1\dots w_n }.\]
The coefficient $\deg_1(\eta)$ is determined by the repetitions of branches defining $\eta$ and $[\omega]_n$ stands for the coefficient of the $n$-th degree term of $\omega$. The only non-trivial task is to obtain the Thom polynomial $tp(\eta)$ in the $c_k$ and $s_0$, which can be acomplished based on works Rim\'anyi \cite{Rimanyi:2001,Rimanyi2002Multiple-Point-} and Kazarian \cite{Kazarian2003Multising,Kazarian:2006}.

For completeness, we start with the formulas for $n\leq 3$, due to Ohmoto. 
The only $0$-stable invariant for $n=1$ is the number of double points:
\begin{enumerate}
\item[]$\# A_0^2= \dfrac{s_0-c_1}{2! \sigma_1}$.
 \end{enumerate}
For $n=2$, the number of triple points and cross-caps are:
\begin{enumerate}
\item[]$\# A_0^3=\dfrac{s_0^2-3s_0c_1+2c_1^2+2c_2}{3! \sigma_2}$,
\vspace{2mm}
\item[]$\# A_1=\dfrac{c_2}{\sigma_2}$.
\end{enumerate}
Finally, for $n=3$ the invariants are the number of quadruple points and number of transverse incidences of a curve of cross-caps with a regular branch:
\begin{enumerate}
\item[]$\# A_0^4= \dfrac{1}{4! \sigma_3}(s_0^3-6s_0^2c_1+11s_0c_1^2+8s_0c_2-6c_1^3-18c_1c_2-12c_3)$,
\vspace{2mm}
\item[]$\# A_0A_1=  \dfrac{1}{\sigma_3}(s_0c_2-2c_1c_2-2c_3)$.
 \end{enumerate}
 
The invariants for $n=4$, are the number of quintuple points, the incidence of two regular branches and surface of cross-caps, and the number of $A_2$ singularities:
\begin{enumerate}
\item[]$\# A_0^5 = \dfrac{1}{5!\sigma_4}(s_0^4-10s_0^3c_1+35s_0^2c_1^2+20s_0^2c_2-50s_0c_1^3-110s_0c_1c_2 \vspace{1.5mm}  \\ 
\phantom{\# A_0^5}-60s_0c_3+ 24c_1^4+144c_1^2c_2+216c_1c_3+48c_2^2+144c_4 )$
\vspace{2mm}
\item[]$\#A_0^2A_1 = \dfrac{1}{2!\sigma_4}\big(s_0^2 c_2 - 5 s_0 c_1 c_2- 4 s_0 c_3 + 6 c_1^2 c_2 + 14 c_1 c_3+ 4 c_2^2  +12 c_4\big)$
\vspace{2mm}
\item[]$\#A_2= \dfrac{1}{\sigma_4}\big(c_1 c_3 + c_2^2 + 2 c_4\big)$.
\end{enumerate}
The Thom polynomials which lead to the above expressions are obtained by dividing Kazarian's polynomials $m_{\eta}$ by a certain correction coefficient, see Section 2 and 5 of \cite{Kazarian2003Multising} for details.

For $n=5$, the invariants are the incidence of three regular branches and a 3-space of cross-caps, the incidence a regular branch with a curve of $A_2$, and the incidence of two three-spaces of cross-caps:
\begin{enumerate}
\item[]$\# A_0^6= \dfrac{1}{6!\sigma_5}\Big( s_0^5 -  15s_0^4 c_1 +5 s_0^3(17 c_1^2 + 8 c_2)  -   15s_0^2 (15 c_1^3+ 26 c_1 c_2 + 12 c_3)  \vspace{1.5mm}\\ 
\phantom{\# A_0^6}     + 2s_0 (137 c_1^4 + 607 c_1^2 c_2 + 164 c_2^2 + 738 c_1 c_3 + 432 c_4)  \vspace{1.5mm}\\
\phantom{\# A_0^6}      
-120 (c_1^5 + 10 c_1^3 c_2 + 10 c_1 c_2^2 + 25 c_1^2 c_3 + 12 c_2 c_3 + 
     38 c_1 c_4 + 24 c_5)\Big),$
\vspace{2mm}  
\item[]$\#A_0^3A_1= \dfrac{1}{3!\sigma_5}\Big((s_0^3c_2 - 3 s_0^2
 (3 c_1 c_2 + 2 c_3) + 
 2 s_0 (13 c_1^2 c_2 + 7 c_2^2 + 24 c_1 c_3 \vspace{1.5mm}\\
\phantom{\# A_0^3A_1} + 18 c_4)  -24 (c_1^3 c_2 + 4 c_1^2 c_3 + 3 c_2 c_3 + 2 c_1 (c_2^2 + 4 c_4) + 6 c_5)\Big)$,
\vspace{2mm}  
\item[]$\#A_0A_2= \dfrac{1}{\sigma_5}\Big(s_0(c_2^2 + c_1 c_3 + 2 c_4)-3 (c_1^2 c_3 + 2 c_2 c_3 + c_1 (c_2^2 + 4 c_4) + 4 c_5)\Big),$
\vspace{2mm}
\item[]$\#A_1^2= \dfrac{1}{\sigma_5}( s_0c_2^2-2 c_1^2 c_3-4 c_1 c_2^2 -8 c_2 c_3 -10 c_1 c_4 -12 c_5)$.
      \end{enumerate}
In this case, the polynomials $m_{\eta}$ cannot be found in Kazarian's paper; they are obtained by putting together Theorem 5.3 and the corresponding ingredients from the lists about residual classes $R_{\eta}$ and classes $n_{\eta}$ from \cite{Kazarian2003Multising}.

\subsection*{Relations between $\boldsymbol 0$-stable invariants in corank 1}
The corank of a map germ $f\colon (\C^n,0)\to(\C^{n+1},0)$ is 
\[\mathrm{corank}(f)=n- \mathrm{rank}(\mathrm df_0).\]
To better relate the corank with the weights and degrees, we restrict ourselves to map-germs in \emph{normal form}, that is, for a given $\mathcal A$-class of rank $r$, we only consider the representatives $F\colon (\C^{n}\times\C^r,0)\to (\C^{n+1}\times\C^r,0)$ of the form 
\[(z,y)\mapsto(f_y(z),y), \quad y\in \C^r.\]
Consider the germ \[f_{y_1,\dots,y_k}\colon (\C^{n+k},0)\to(\C^{n+1+k},0),\]
obtained by making the parameters $y_{k+1},\dots,y_{r}$ equal to zero.
To such an $F$, and a $0$-stable type $\eta\colon \C^{n+k}\to\C^{n+1+k}$, we associate the number \[\#\eta=\#\eta(f_{y_1,\dots,y_k})\]
if $f_{y_1,\dots,y_k}$ is $\mathcal{A}$-finite, and $\#\eta=\infty$ otherwise.

Observe that only the $\#\eta$ of stable types $\eta\colon\C^{n+r}\to\C^{n+1+r}$ are $\mathcal A$-invariants of $F$. The numbers $\#\eta$ for lower dimensional $\eta$ are just numbers that come in handy.

\begin{propo}\label{propoRels0InvSlice}
Let $F\colon (\C^{n},0)\to (\C^{n+1},0)$ be an $\mathcal A$-finite weighted-homogeneous map-germ of corank 1 in normal form and $f_{y_1,\dots,y_{n-2}}$ is also $\mathcal{A}$-finite. If $n\leq 5$, then
\begin{align*}
&\#A_0^{n+1}(F)=\# A_0^{n}(f_{y_1,\dots,y_{n-2}})\dfrac{(d_0-n w_1)(d_1-n w_1)}{n w_1 w_n}\\ 
&\#A_0^{n-2}A_1(F)=\# A_0^{n-3}A_1(f_{y_1,\dots,y_{n-2}}) \dfrac{(d_0-(n-1) w_1)(d_1-(n-1) w_1)}{(n-2) w_1 w_n}\\
& \#A_0^{n-2}A_1(F)= \# A_0^{n}(f_{y_1,\dots,y_{n-2}})\dfrac{n(n-1)w_1}{w_n}\\
&\#A_0^{n-4}A_2(F)= \# A_0^{n-3}A_1(f_{y_1,\dots,y_{n-2}})\dfrac{(n-3)w_1}{w_n}\\
&\#A_0A_2(F)= 2\# A_1^{2}(f_{y_1,\dots,y_{n-2}}).
\end{align*}
\end{propo}

\begin{re}
These relations do not hold in higher corank, as the singularities $\hat P_1$ and $\hat N_1$ from Table \ref{SharlandC5C6} show. One checks that $\#A_1^2(\hat P_1)=2$, $\#A_0A_2(\hat P_1)=6$, $\#A_1^2(\hat N_1)=40$ and $\#A_0A_2(\hat N_1)=84$.
\end{re}

The arrows in the following diagram indicate that, for map-germs of corank 1, the vanishing of the source forces the vanishing of the target.
\begin{equation}
\xymatrix{
A_0^2 \ar@{->}[r] \ar@{<->}[rd] & A_0^3 \ar@{->}[r] \ar@{<->}[rd] & A_0^4 \ar@{->}[r]  \ar@{<->}[rd] & A_0^5\ar@{->}[r] \ar@{<->}[rd] & A_0^6 \\
 &  A_1  \ar@{->}[r]  &  A_0 A_1 \ar@{->}[r]  &  A_0^2 A_1 \ar@{->}[r]   \ar@{<->}[rd]&  A_0^3 A_1\\ 
& & &  A_2 \ar@{<->}[ul]\ar@{->}[r]&  A_0 A_2 \\
& & & & A_1^2 \ar@{<->}[u]}\label{D}
\end{equation}

This relations come as no surprise, since the invariants $\#\eta(F)$ are expressed in $c_{k,n}$ and $s_{0,n}$ and, as we will see, these behave well under unfoldings. Let $F$ be a weighted-homogeneous $1$-parameter unfolding, and hence $d_n=w_n$. From the geometric construction giving rise to the functions $\sigma, s_0$ and $c_k$,  it follows easily that 
\begin{equation*}c_{k,n}(w,d)=c_{k,n-1}(w_1,\dots,w_{n-1},d_0,\dots,d_{n-1}),\quad \text{for all } k,\label{eqChernSlice}\end{equation*}
and that $s_{0,n}(w,d)=s_{0,n-1}(w_1,\dots,w_{n-1},d_0,\dots,d_{n-1})$.
This suggests the existence of a function $q(w,d)$ in variables $w$ and $d$ satisfying \[c_{k,n}(w,d)=c_{k,n-1}(w_1,\dots,w_{n-1},d_0,\dots,d_{n-1})+(d_n-w_n)q(w,d).\]
 Indeed, a little combinatorics shows the following:
 \begin{propo}\label{ChernUnfolding}For any positive integer $n$, we have
  \[c_{k,n}=c_{k,n-1}+(d_n-w_n)\sum_{i=0}^{k-1}(-1)^iw_n^ic_{k-i-1,n-1}\]
  and
  \[s_{0,n}=s_{0,n-1}\frac{d_n}{w_n}.\]
 \end{propo}

\section{Ohmoto's image Milnor number formulas}\label{secOhmotoFormula}

We give some theoretical background from \cite{Ohmoto:2014} and explain our methods used to prove  formulas in Theorem \ref{FormulaNew}. The method is illustrated in detail for $n=2$ and briefly for $n=3$.

\subsection*{Characteristic classes and image Milnor number}
Our interpolation method is based entirely on Proposition \ref{thmFormulaOhmoto}, which is an immediate consequence of Ohmoto's Theorem \ref{ThmOhmoto}.  The reader should bear in mind that Proposition \ref{thmFormulaOhmoto} expresses $\mu_I$ in terms of weights and degrees, exclusively. This allows for the interpolation method to be applied blindly, independently of its origins in the theory of characteristic classes.  

 For $\alpha=(\alpha_0,\dots,\alpha_n)$ we let $\| \alpha\|=\alpha_0+\sum_{k=1}^n k \alpha_k$ and $c^{\alpha}=s_0^{\alpha_{0}}c_1^{\alpha_{1}}\ldots c_n^{\alpha_{n}}$.

\begin{propo}\label{thmFormulaOhmoto} There are unique $b_{\alpha}\in \mathbb{Q}$, with $0\neq\alpha\in\mathbb{N}^{6}$, such that any $\mathcal{A}$-finite weighted homogeneous map-germ $F\colon (\C^n,0)\rightarrow(\C^{n+1},0)$, with $n\leq 5$, satisfies
\begin{equation}\label{F}
\mu_I(F)=(-1)^n\frac{\sum_{\|\alpha\|\leq n} b_{\alpha} c^{\alpha}\sigma_{n-\| \alpha\|}}{\sigma_n}.
\end{equation}
\end{propo}
The remaining of the section is devoted to explaining how Proposition \ref{thmFormulaOhmoto}  derives immediately from the following result.

\begin{teo}\label{ThmOhmoto}  \cite[Theorem 6.20]{Ohmoto:2014}
Let $F\colon(\C^n,0)\to(\C^{n+1},0)$ be $\mathcal A$-finite and $n\leq 5$. The Euler characteristic of the image of a stable perturbation $F_y$ of $F$ is
\begin{equation}\label{FOh}
\chi(\Im(F_y))=\frac{[c(E_0)\cdot tp^{\text{SM}}(\alpha_{\text{image}})(c(F))]_{n}}{[c_n(E_0)]_n}.
\end{equation}
\end{teo}

Observe that the left hand side of Equation (\ref{F}) is $\mu_I(F)=(-1)^n(\chi(\Im(F_y))-1)$. Now we proceed to describe the ingredients in the right hand side of Formula (\ref{FOh}).

Let $\ell$ be the dual tautological line bundle over $\PP^\infty$. Associated to the grading $(w,d)$, there are the two bundles 
\[E_0:=\ell^{\otimes w_1}\oplus\dots\oplus \ell^{\otimes w_n}\quad\text{and}\quad E_1:=\ell^{\otimes d_0}\oplus\dots\oplus \ell^{\otimes d_n}.\]

The cohomology of $\PP^\infty$ is isomorphic to the polynomial ring $\mathbb Z[a]$ and, under this isomorphism, the total Chern class of $\ell$ is $c(\ell)=1+a$. From this we obtain the total Chern classes 
\[c(E_0)=\prod_{j=1}^n(1+w_j a)\quad \text{and}\quad c(E_1)=\prod_{i=0}^n(1+d_i a).\]

One can construct a \emph{universal map}  $\tilde{F}\colon E_0\to E_1$ whose restriction to a each fiber is $\mathcal{A}$-equivalent to $F$ \cite{Ohmoto:2014}. By abuse of notation, one writes $c(F)$ for the total chern class $c(\tilde{F})=c(\tilde{F}^*TE_1-TE_0)$ of the virtual normal bundle. One checks that
\[c(F)=\frac{\prod_{i=0}^n(1+d_i a)}{\prod_{j=1}^n(1+w_j a)}.\]
The functions $\sigma_j(w,d)$, $\delta_i(w,d)$ and $c_k(w,d)$ from Section \ref{RESULTS} are precisely the coefficients in the graded decompositions
\begin{align*}
c(E_0)&=1+\sigma_1a+\dots+\sigma_na^n,\\
c(E_1)&=1+\delta_1a+\dots+\delta_{n+1}a^{n+1},\\
c(F)&=1+c_1 a+c_2 a^2+c_3a^3\dots.
\end{align*}

 The term $tp^{\text{SM}}(\alpha_{\text{image}})$ is the Segre-MacPherson Thom polynomial of the constructible function $\alpha_{\text{image}}$. This is an extension of the classical Thom  polynomial in the following sense:

The classical Thom polynomial \cite{Thom:1956} of a stable mono-singularity type $\eta\colon (\C^n,0)\to(\C^{n+k},0)$ is the unique polynomial $tp(\eta)\in \Z[c_1,c_2,\dots]$ such that, for any stable map $f\colon M^m\to N^{m+k}$, the following holds
\[\text{Dual}[\overline{\eta(f)}]=tp(\eta)(c(f)).\]
The left hand side of the equality is the Poincaré dual of the fundamental class of the closure of the singularities of type $\eta$ exhibited by $f$. The right hand side is the evaluation of $tp(\eta)$ in the total Chern class $c(f)=c(f^*TN-TM)$ of the virtual normal bundle of $f$.

 For multi-singularity types $\underline{\eta}$, the definition was extended by Kazarian \cite{Kazarian2003Multising} to Thom polynomials $tp(\underline{\eta})$
 depending on further variables $s_I$. They satisfy the analogous property $\text{Dual}[\overline{\underline{\eta}(f)}]=tp(\underline{\eta})(s_I(f),c(f))$. Here, the class $s_I(f)$ is the \emph{Landweber-Novikov class} of $f$, it is defined by
\[s_I(f)=f^* f_*(c_1(f)^{i_1}c_2(f)^{i_2}\dots),\] for the multi-index $I=(i_1i_2\dots)$. The 0-th Landweber-Novikov class of $f$ is $s_0(f)=f^*f_*(1)=\frac{c_{top}(f^*TN)}{c_{top}(TM)}$. For simplicity, we consider the evaluation of $tp(\overline{\eta})$ in the Chern classes $c_i(f)$ and the class $s_0(f)$. In particular, the universal map above gives \[s_0(F)=\frac{c_{n+1}(E_1)}{c_n(E_0)}=s_0a.\]

The Segre-MacPherson Thom polynomial $tp^{\text{SM}}(\underline{\eta})$ is the unique series in $s_0$, $c_i$ satisfying the following similar property \cite{Ohmoto:2014}: Let $i\colon X\hookrightarrow M$ be a closed embedding of an algebraic subvariety $X$ in an algebraic manifold $M$. Consider $\mathcal{F}(X)$ the abelian ring of constructible functions on $X$. The Chern-Schwartz-MacPherson natural transformation \cite{MacPherson1974} determines a morphism $c_*\colon \mathcal{F}(X)\to H_*(X)$. The Segre-Schwartz-MacPherson class of the embedding $i$ is 
\[s^{\text{SM}}(X,M):=c(i^*TM)^{-1}\cap c_*(\mathbbm{1}_X).\]
The above mentioned property satisfied by $tp^{\text{SM}}(\underline{\eta})$ is  \[\text{Dual }(i_*s^{\text{SM}}(\overline{\underline{\eta}(f)},M))=tp^{\text{SM}}(\underline{\eta})(s_0(f),c(f)).\]

The Segre-MacPherson Thom polynomial has the form $tp^{\text{SM}}(\eta)=tp(\eta)+ \text{ higher degree terms}$. 
The definition of Segre-MacPherson Thom polynomials extends to certain constructible functions $\alpha$, so that $tp^{\text{SM}}(\mathbbm{1_{\underline \eta}})=tp^{\text{SM}}({\underline \eta})$. For $f\colon M\to N$, there is such a constructible function $\alpha_{image}$,  determined by $\mathbbm{1}_{f(M)}=f_*(\alpha_{image})$.

Finally, $[c(E_0)\cdot tp^{\text{SM}}(\alpha_{\text{image}})(s_0(F),c(F))]_{n}$ stands for the $n$-th degree part of the series $c(E_0)\cdot tp^{\text{SM}}(\alpha_{\text{image}})(s_0(F),c(F))$ in the variable $a$ (note that, by abuse of notation, the term $tp^{\text{SM}}(\alpha_{\text{image}})(s_0(F),c(F))$ appears in \cite{Ohmoto:2014} as $ tp^{\text{SM}}(\alpha_{\text{image}})$).

\subsection*{How to obtain the $\boldsymbol{\mu_I}$ formulas} 
We write the multi-indices $\alpha\in \mathbb N^6$ of Proposition \ref{thmFormulaOhmoto}  only up to their last non zero entry. 
For example, we write $(0,1)$ for $\alpha=(0,1,0,0,0,0)$.

Our strategy is based on the following simple interpolation idea: pick any $\mathcal A$-finite map-germ $F$, with known image Milnor number. Every possible weigths $w$ and degrees $d$ of $F$ determine values $\sigma_k$, $s_0$ and $c_k$, and the formula  yields a linear equation in the variables $b_\alpha$. The data \[\tau(F):=(w,d,\mu_I(F))\] will be called a \emph{sample} of $F$.
The $b_\alpha$ are determined after sampling singularities as many times as the number of $b_\alpha$, provided that each sample gives an equation which is independent from the preceding ones. 

The initial cases are somehow trivial, because the literature contains enough singularities to complete the interpolation. The difficulties arise in $n\geq 4$, due to the lack of $\mathcal A$-finite map-germs with known $\mu_I$. The key points of our interpolation strategy are contained in Remarks \ref{DifferentSamples} and \ref{ZeroInvariants}, Proposition \ref{EquationStable} and  \ref{NoFirstExamples}.

Ohmoto's formula for $n=2$ (see \cite[Example 6.21]{Ohmoto:2014}) can be rewriten as
 \[\mu_I(F)=\dfrac{1}{\sigma_2}\big(\dfrac{1}{2!}(-s_0+c_1) \sigma_1+
\dfrac{1}{3!}(s_0^2-c_1^2-c_2)\big).\]
To recover the formula, we need to determine the six $b_{\alpha}$ with $\| \alpha \| \leq 2$, hence we must find six samples giving rise to independent equations. Again, all maps will be chosen in normal form. Since the computations involved often become hard, we want to sample the simplest singularities first. 

We start in corank 0, that is, with $d_1=w_1$ and $d_2=w_2$. Since corank 0 map-germs are regular, we know  $\mu_I(F)=0$. Replacing $d_1$ and $d_2$ by $w_1$ and $w_2$, the formula (\ref{F}) reads
\begin{equation} \label{FormulaCorank0}
0=\dfrac{d_0((b_{02}+b_{11}+b_{2})d_0+(b_{01}+b_{1})(w_1+w_2))}{w_1w_2}.
\end{equation}

The regular map 
\[R\colon (z,y)\mapsto (0,z,y)\]
admits samples $\tau_1(R)=((1,1),(1,1,1),0)$ and $\tau_2(R)=((1,1),(2,1,1),0)$. By substitution, it follows
\[b_{02}+b_{11}+b_{2}=b_{01}+b_{1}=0.\]

\begin{re}\label{DifferentSamples}
\emph{Different samples of a same singularity may give independent equations}. Observe that 
1) this is exactly what happened with $\tau_1(R)$ and $\tau_2(R)$, and it will continue to happen for all higher degrees. However,
2) let $(w,d)$ and $(w',d')$ be two gradings of a map-germ $F$. If $(w',d')=\lambda (w,d)$, for $\lambda\in\mathbb Q$, then the samples $\tau_1(F)=(w,d,\mu_I(F))$ and $\tau_2(F)=(w',d',\mu_I(F))$ give rise to the same equation on $b_{\alpha}$. This is because the coefficients of each $b_{\alpha}$ in Formula (\ref{F}) are homogeneous rational functions of degree zero in the weights and degrees. 3)  Two representatives $F$ and $F'$ of the same $\mathcal A$-class may produce different sets of samples. For example, the representative $(z,z,y)$ in the $\mathcal{A}$-class of $R$ admits $\tau_1(R)$ as a sample but not $\tau_2(R)$. The map $(z^2,z,y)$ admits the opposite combination. An strategy to find better representatives is to eliminate monomials in the coordinate functions. For instance, the cross-cap $(z^2,z^3+yz,y)$ is $2$-determined and by eliminating the $z^3$ term, the resulting representative admits more samples. 
\end{re}

We claim that no further independent equations can be found by sampling the regular map $R$. This is because for any $b_\alpha$, satisfying the above equations, the right hand side of equation (\ref{FormulaCorank0}) vanishes. Having exhausted $R$, we must move to the case where $d_1\neq w_1$, and the simpler such singularities are the map-germs of corank 1. 

Every singular map-germ has $\#A_1>0$, but we still want to start with the simpler ones, having $\#A_0^3=0$ and $\#A_1$ as low as possible. Consider the cross-cap, parameterised as
\[A_1\colon (z,y)\mapsto(z^2,yz,y),\]
and the samples
\[\tau_i(A_1)=((1,i),(2,i+1,i),0).\]
The samples $\tau_1(A_1)$ and $\tau_2(A_1)$ give equations 
\[0=8b_{1}+6b_{01}+16b_{2}+12b_{11}+9b_{02}+b_{001}\]
and 
\[0=9b_{1}+6b_{01}+18b_{2}+12b_{11}+8 b_{02}+b_{001},\]
respectively.
We show now that no more independent equations can be obtained from map-germs having $\#A_1=1$ and $\#A_0^3=0$. The idea is to look at the expressions of these invariants for corank 1 germs:
\[\#A_1=\frac{(d_0-w_1)(d_1-w_1)}{w_1w_2},\quad \#A_0^3=\frac{(d_0-2w_1)(d_1-2w1)}{6w_1^2}\#A_1.\]
If $\#A_1$ does not vanish, the condition $\#A_0^3=0$ implies $d_0=2w_1$ or $d_1=2w_1$ and, by a permutation of the coordinate functions of the map-germ, we may choose $d_1=2w_1$.  Replacing $d_1$ by $2w_1$  in the expression $\# A_1=1$, we obtain $w_2=d_0-w_1$. Eliminating four of the $b_\alpha$ by means of the previous equations and impossing the conditions $d_1=2w_1$ and $d_2=w_2=d_0-w_1$, we obtain a closed expression for $\mu_I$, independent of the remaining $b_{\alpha}$. This means that the last two $b_{\alpha}$ cannot be found by taking samples satisfying such conditions.  This illustrates another key point of the interpolation strategy:

\begin{re}\label{ZeroInvariants} \emph{The numbers $\#\eta$ may be used to decide whether a singularity has been exhausted.}
\end{re}

Since the cross-cap is the only singular stable mono-germ for dimensions $(2,3)$, from now on we need to take non-stable map-germs into account. A known singularity with $\# A_1=2$ and $\#A_0^3=0$ is 
\[S_1\colon (z,y)\mapsto(z^2,z^3+y^2z,y).\]
It is well known that $S_1$ has $\mathcal A_e$-codimension one and, since Mond's conjecture holds for $n\leq 2$ (see \cite{Mond:1995,Jong:1991}), this number is precisely $\mu_I(S_1)$.

After one sampling of $S_1$, one checks easily that we need a sample with $\#A_0^3\neq 0$.
The interpolation is finished after sampling Mond's map-germ
\[H_2\colon (z,y)\mapsto (z^3, z^5+yz,y),\] 
which has $\mu_I(H_2)=2$.

Table \ref{Invariants2} contains numbers associated to the interpolation samples.  Horizontal lines separate changes in corank. The number $d_0$ is only included for $\tau_1(R)$ and $\tau_2(R)$, because it does not carry any clear geometric information about the rest of singularities. The $\infty$ symbol means that $\# A_0^2$ is not well defined for the corresponding slice. For higher $n$, there will be too many associated numbers, and we will include only the essential ones, based on Diagram (\ref{D}).
\begin{table}[ht]
\setlength{\extrarowheight}{2.5pt}
\begin{center}
\begin{tabular}{ c  c  c  c  c }
Sample & $d_0$ & $\#A_0^2$ & $\#A_1$  & $\#A_0^3$ \\
\hline 
$\tau_i(R)$, $i=1,2$ & i & 0 & 0 & 0 \\ \hline
$\tau_i(A_{1}), i=1,2$ & & $\infty$ & 1 & 0  \\
 $\tau(S_1)$  &  & 1 & 2 & 0  \\
 $\tau(H_2)$  &  & 4 & 2 &1 \\ \hline
\end{tabular}
\end{center}
\caption{Numbers associated to the samples for $n=2$.}\label{Invariants2}
\end{table}

The $\mu_I$ formula for $(\C^3,0)\to (\C^4,0)$ (see \cite[Example 6.22]{Ohmoto:2014}) reads
\begin{align*}
\mu_I(F)&=-\frac{1}{\sigma_3}\big(\frac{1}{2!}(-s_0+c_1) \sigma_2+\frac{1}{3!}(s_0^2-c_1^2-c_2)\sigma_1+\\
&+\frac{1}{4!}(-s_0^3-2 s_0^2c_1+ s_0 c_1^2+16 s_0 c_2+2c_1^3-10c_1c_2) \big).
\end{align*}

Observe that, from the $b_{\alpha}$ with $\|\alpha\|\leq 3$ to be found, we already know the ones with  $\|\alpha\|\leq 2$. This leaves us with the seven unknown  $b_{\alpha}$. 

The first equations are obtained automatically from the following result, based on sampling trivial unfoldings of stable singularities of smaller dimensions.

\begin{propo}\label{EquationStable}
Let $F\colon (\C^n,0)\rightarrow (\C^{n+1},0)$ be a weighted-homogeneous stable map-germ and let $(w,d)$ be a grading of $F$. With the notations above, the coefficients $b_{\alpha}$ satisfy
\[0=\sum_{\| \alpha \|\leq n+r}b_{\alpha} c^{\alpha}\Big(\sum_{k=0}^{n+r-\|\alpha\|}\binom{k}{r}\sigma_{n+r-\|\alpha\|-k}\Big),\]
for all $r\geq 0$.
\end{propo}
\begin{proof}
Observe that any trivial $r$-parameter unfolding of $F$ is also stable and admits the grading $((w,1,\dots,1),(d,1,\dots,1))$. The result follows putting together Proposition \ref{thmFormulaOhmoto}, Proposition  \ref{ChernUnfolding} and the equality
\[\sigma_{\ell,n+r}(w,1,\dots,1)=\sum_{k=0}^{\ell}\binom{k}{r}\sigma_{\ell-k}(w).\qedhere\]
\end{proof}

Applying this property to the $(w,d)$ from $\tau_1(R)$, $\tau_1(A_1)$ and $\tau_2(A_1)$ with $r=1$, we obtain three independent equations.

Notice that, for $r=1$, the equation from $\tau_2(R)$ is not independent of the one from $\tau_1(R)$. The sample $\tau_3(A_1)$ which did not produce an independent equation for $n=2$, does give a new equation for $n=3$, that is, for $r=1$. We have used the singularities from Table \ref{TableSingularities3} to finish the interpolation. That is, Houston's and Kirk's singularities $P_1$ and $P_2$ \cite{Houston:1999}, and Sharland's  singularity $\hat B_3$ \cite{Altintas:2014}.

\begin{table}[ht]
\setlength{\extrarowheight}{2.5pt}
\begin{center}
\begin{tabular}{ c  c  c }
Label & Map-germ & $\mu_I$ \\
\hline 
 $P_1$  & $(z(z^3+y),z(z^2+x),y,x)$ & 1  \\
 $P_2$  & $(z(z^4+y),z(z^2+x),y,x)$ & 2 \\
 $\hat B_3$ & $(y^2+xz,z^2-xy,y(y^2+z^2)+z(y^2-z^2),x)$ & 33 \\ \hline
\end{tabular}
\end{center}
\caption{Singularities sampled beside $R$ and $A_1$ for $n=3$.}\label{TableSingularities3}
\end{table}

From the diagram (\ref{D}), it follows that the crucial invariants are the ones in the following table:

\begin{table}[ht]
\setlength{\extrarowheight}{2.5pt}
\begin{center}
\begin{tabular}{ c  c  c  c  c }
Sample & $\#A_1$ & $\#A_0A_1$  & $\#A_0^4$ \\
\hline 
$\tau_1(R)$ & 0 & 0 & 0 \\ \hline
$\tau_i(A_1)$, $i=1,2,3$  & 1 & 0 & 0 \\
 $\tau(P_1)$  & 3 & 2 & 0  \\
 $\tau(P_2)$  & 2 & 3 & 0 \\ \hline
 $\tau(\hat B_3)$ & 5 & 16 &1 \\ \hline
\end{tabular}
\end{center}
\caption{Numbers associated to the samples for $n=3$.}\label{Invariants3}
\end{table}

Notice that the singularities in Table \ref{TableSingularities3} have corank 1, with the exception of Sharland's singularity $\hat B_3$ of corank 2. However, the interpolation method can be completed without resort to corank 2 map-germs.  We may use the singularity 
\[F\colon (z,y,x)\mapsto (z^4-xz,(y+z)^5+xz^2,y,x).\]
The previous map-germ of corank 2 was included in order to avoid justifying that $F$ is $\mathcal A$-finite with $\mu_I(F)=52$. Criteria for $\mathcal A$-finiteness and computation of $\mu_I$ will be discussed in Section \ref{Afinite}.

\section{Ohmoto's $\mu_I$ formulas for $n=4$, $5$}\label{SecNew}
Here we sketch the steps to prove Theorem \ref{FormulaNew}. The same interpolation idea used for $n=2,3$ applies just as fine for $n=4,5$ but, as far we know, the examples found in the literature do not suffice to complete the associated system of equations. 

As it turns out, it is not always easy to produce $\mathcal A$-finite singularities giving new  independent equations. One has to bear in mind that checking $\mathcal A$-finiteness and computing $\mu_I$ are often computationally unfeasible tasks. For $\mathcal A$-finiteness, we use a geometric criteria based on multiple points. For $\mu_I$, we first compute the $\mathcal A_e$-codimension (for which commutative algebra algorithms exist), then we justify that the germ satisfies Mond's conjecture, ensuring the equality of $\mu_I$ and the computed $\mathcal A_e$-codimension.  

\subsection*{The $\boldsymbol{\mu_I}$ formula for $\boldsymbol{n=4}$} 

From Proposition \ref{EquationStable} applied to $R$ and $A_1$, for $r=2$, we find five independent equations from $\tau_1(R)$ and $\tau_i(A_1)$, for $i=1,\dots,4$.
One can check that no more samples from singularities $\#A_1=1$ can be used. 

Our next move is to consider the stable singularity
\[A_2\colon (z,y,x,t)\mapsto(z^3+tz,yz^2+xz,y,x,t)\]
with samples 
\[\tau_i(A_2)=((1,i+1,2,i),(3,i+2,i+1,2,i),0).\]
Three new equations arise from $\tau_1(A_2)$, $\tau_2(A_2)$ and $\tau_3(A_2)$.
One checks that map-germs with $\#A_2\leq 1$ and $\#A_0^2A_1=0$ do not provide new information. Also, nothing new comes from map-germs that the authors could find in the literature.
We consider the new map-germs 
\begin{itemize}
\item[] $L_1\colon (z,y,x,t)\mapsto(z^4-tz,(y+z)^6+xz,y,x,t),$
\vspace{2mm}
\item[] $L_2\colon (z,y,x,t)\mapsto(z^4+xz^2+tz,(y+z)^5+(x^2+ty)z,y,x,t),$
\end{itemize}
which have $\#A_0^5=0$ and $\mu_I(L_1)=39$ and $\mu_I(L_2)=87$. At this point, it is not possible to obtain further equations if $\# A_0^5=0$. We take another map-germ 
\[L_3\colon (z,y,x,t)\mapsto(z^5-tz,(y+z)^7+xz,y,x,t),\]
whith $\mu_I(L_3)=178$. To avoid disrupting the flow of the explanation, the $\mu_I$ values and $\mathcal A$-finiteness of $L_1$, $L_2$ and $L_3$ will be justified in Section \ref{Afinite}. The singularities $L_i$ were not our first candidates for the interpolation. In Remark \ref{NoFirstExamples} we  explain what brought us to them. 

At this stage, one can check that it is necessary to introduce map-germs of corank 2. For instance, Sharland's
\[\hat D_1 \colon (z,y,x,t)\mapsto(y^2+xz+(x^2+t)y,yz,z^2+y^3+t^2y,x,t),\]
which is known to have $\mu_I(\hat D_1)=27$ \cite{Altintas:2014}. This finishes the proof of Theorem \ref{FormulaNew} for $n=4$, except from the claimed $\mathcal A$-finiteness and image Milnor numbers of $L_1$, $L_2$ and $L_3$. 

\begin{re}
If one does not care about introducing more map-germs of corank 2, $L_1$ and $L_2$ can be interchanged by Sharland's $\hat E_1$ and $\hat K_1$ \cite{Altintas:2014}. It is however unavoidable to study the $\mathcal A$-finiteness and the $\mu_I$ of at least one new map-germ. This is because the system of equations cannot be closed without resorting to a map-germ with $\A_0^5\neq 0$, such as $L_3$.
\end{re}

\begin{table}[ht]
\setlength{\extrarowheight}{2.5pt}
\begin{center}
\begin{tabular}{ c  c  c  c  c}
Sample & $\#A_1$  & $\#A_2$ & $\#A_0^2A_1$ & $\#A_0^5$ \\
\hline 
$\tau_1(R)$  & 0 & 0 & 0 &0 \\ \hline
$\tau_i(A_1)$, $i=1,\dots,4$  & 1 & 0 & 0 & 0 \\
 $\tau_i(A_2)$, $i=1,2,3$  & $\infty$ & 1 & 0 & 0 \\
 $L_1$  & 15 & 8 & 12 & 0 \\
 $L_2$  & 12 & 12 & 12 & 0 \\
 $L_3$  & 24 & 15 & 60 & 3 \\ \hline
 $\hat D_1$  &$\infty$ & 9 & 0 & 0 \\
\hline
\end{tabular}
\end{center}
\caption{Numbers associated to the samples for $n=4$.}\label{Invariants4}
\end{table}

\subsection*{The $\boldsymbol{\mu_I}$ formula for $\boldsymbol{n=5}$} There are 19 unknown $b_{\alpha}$ to be determined. We will need six new map-germs and to stablish their $\mathcal A$-finiteness and $\mu_I$ values.

Again, Proposition \ref{EquationStable} is applied a number of times, in this case to
$\tau_1(R)$, $\tau_i(A_1)$, for $i=1,\dots, 5$, and $\tau_i(A_2)$, for $i=1,\dots,4$.
Next samples need to satisfy $\#A_2>1$ and hence they cannot be stable. By a similar argument, at least three map-germs of corank 2 will be necessary to close the formula. These will be Sharland's map-germs \cite{Altintas:2014}, $\hat M_{1,1}$, $\hat P_{1}$ and $\hat N_{1}$\footnote{There seems to be a typo in Sharland's parameterisation of $\hat 
N_1$. Our term $x^4y$ replaces her $x^2y$, inconsistent with the claim that $\hat N_1$ unfolds $\hat E_1$.} with image Milnor numbers $13$, $24$, $1400$ and coordinate functions as in Table \ref{SharlandC5C6}.

\begin{table}[ht]
\setlength{\extrarowheight}{2.5pt}
\begin{center}
\begin{tabular}{c c }
Label &{Map-germ} \\
 \hline 
$\hat M_{1,1}$&  $(y^2+xz+(x^2+s)y,yz+ty,z^2+y^3+s^2y,x,t,s)$\\
$\hat P_1$&$(y^2+(x+s)z,z^2+xy,y^3+s^2y+z^3+yz^2+tz,x,t,s)$ \\
$\hat N_1$&$(y^3+(x^4+t)y+xz,(y+s)z,z^2+y^5+s^3y^2+(t^2+s^4)y,x,t,s)$\\
 \hline
\end{tabular}
\end{center}
\caption{Sharland's singularities of corank 2.}\label{SharlandC5C6}
\end{table}

 Once $\hat M_{1,1}$, $\hat P_{1}$ and $\hat N_{1}$ are included, no other known singularity will contribute an independent equation. We produce the  new non-stable map-germs of corank 1 found in Table \ref{NewInterpolation5}, whose $\mathcal A$-finiteness and $\mu_I$ are determined case by case. Again, the singularities $\tilde L_i$ and $Q_i$ were not our first examples for the interpolation. We shall explain the details in the following section.

\begin{table}[ht]
\setlength{\extrarowheight}{2.5pt}
\begin{center}
\begin{tabular}{c c c }
Label &{Map-germ}  & $\mu_I(F)$\\
 \hline 
$\tilde L_2$&$(z^4+tz+xz^2+s^3z,(y+z)^5+x^2z+tyz+s^2z^3,y,x,t,s)$ & {321}\\
$\tilde L_1$&$(z^4-tz+s^2z^2,(y+z)^6+xz+s^3z^3,y,x,t,s)$ & {149}\\
$Q_1$&$(z^4+tz^2+tyz+s^3z,(y+z)^7-xz+s^4z^3,y,x,t,s)$ &{711} \\
$Q_2$& $(z^5-xz+tz^2+s^2z,(y+z)^5+sz^3+xz,y,x,t,s)$ & 144\\
$Q_3$&  $(z^5+(x^2+t)z-sz^2+xz^3,(y+z)^6+sxz-tz^2,y,x,t,s)$ & 654\\
$Q_4$& $(z^8-xz+syz^3,(y+z)^6+tz-sz^2,y,x,t,s)$ &  {862} \\
 \hline
\end{tabular}
\end{center}
\caption{Some new $\mathcal A$-finite singularities.}\label{NewInterpolation5}
\end{table}

This finishes the interpolation for $n=5$. Table \ref{Invariants5} contains the numbers associated to the samples.

\begin{table}[ht]
\setlength{\extrarowheight}{2.5pt}
\begin{center}
\begin{tabular}{ c  c  c  c  c  c}
Sample & $\#A_1$  & $\#A_2$ & $\#A_0A_2$ & $\#A_0^3A_1$ & $\#A_0^6$  \\
\hline 
$\tau_1(R)$  & 0 & 0 & 0 & 0 & 0\\ \hline
$\tau_i(A_1)$, $i=1,\dots,5$  & 1 & 0 & 0 & 0 & 0 \\
 $\tau_i(A_2)$, $i=1,\dots,4$  &$\infty$ & 1 & 0 & 0 & 0\\
 $\tau(\tilde L_2)$  & 12 & 12 & 0 & 0 & 0\\
 $\tau(\tilde L_1)$  & 15 & 8 & 24 & 0 & 0\\ 
 $\tau(Q_1)$  & 18 & 15 & 60 & 0& 0 \\
 $\tau(Q_2)$  & 16 & 12 & 24 & 4& 0 \\
 $\tau(Q_3)$  & 20 & 30 & 60 & 20& 0 \\
 $\tau(Q_4)$  & 35 & 24 & 90 & 120& 3 \\\hline
 $\tau(\hat M_{1,1})$  & 3 & 6 & 0 & 0& 0 \\
 $\tau(\hat P_{1})$  & 5 & 6 & 6 & 0& 0 \\
 $\tau(\hat N_{1})$ & 5 & 33 & 84 & 0& 0 \\
\hline
\end{tabular}
\end{center}
\caption{Numbers associated to the samples for $n=5$.}\label{Invariants5}
\end{table}

\section{$\mathcal A$-finiteness, stabilisations and image Milnor number}\label{Afinite}


The remaining map-germs whose $\mathcal A$-finiteness we must justify have corank 1 and can be studied in terms of their multiple point spaces, thanks to work by Marar and Mond \cite{MararMondCorank1}.

\begin{defi}
Let $F\colon (\C^n,0)\to (\C^p,0)$ be a map-germ of corank one in the normal form $F(z,y)=(f(z,y),y)$, with $y\in \C^{n-1}$ and $z\in \C$. The \emph{$k$-multiple point spaces} $D^k(F)$, are the zero locus 
 in $\C^k\times\C^{n-1}$ of  the iterated divided differences 
  \begin{itemize} 
    \item[] $f_j[z_1,z_2,y]=\dfrac{f_j(z_2,y)-f_j(z_1,y)}{z_2-z_1}$, 
    \item[] $f_j[z_1,z_2,z_3,y]=\dfrac{f_j[z_1,z_3,y]-f_j[z_1,z_2,y]}{z_3-z_2}$
    \item[]$\vdots$
    \item[]$f_j[z_1,\dots,z_k,y]=\dfrac{f_j[z_1,\dots,z_{k-2},z_k,y]-f_j[z_1,\dots,z_{k-2},z_{k-1},y]}{z_k-z _{k-1}},$
  \end{itemize} 
  for $1\leq j\leq p-n+1$.
\end{defi}

\begin{teo}\label{MM}[Marar and Mond \cite{MararMondCorank1}, Theorem 2.14]
With the above notations, 
\begin{enumerate}
\item[$i)$] $F$ is stable if and only if each $D^k(F)$ is smooth of dimension $kn-(k-1)p$. 
\item[$ii)$]$F$ is $\mathcal A$-finite if and only if each $D^k(F)$ is a complete intersection of dimension $kn-(k-1)p$ with isolated singularity, or it is contained in the origin.
\end{enumerate}
\end{teo}

This criterion has been used  for $L_1$, $L_2$, $L_3$, $\tilde L_1$, $\tilde L_2$ and $Q_1,\dots, Q_4$, by means of a \textsc{Singular} \cite{SingularSofware} implementation of the divided differences.

Our methods to compute the image Milnor number require finding stable unfoldings or stabilisations of $F$.
Stable unfoldings are easier to obtain, by means of a well known procedure due to Mather \cite{MatherIV}. We have used stable unfoldings of $L_1$, $L_2$, $L_3$, $\tilde L_1$, $\tilde L_2$ and $Q_1$.

\begin{re}
In certain cases, stable unfoldings are too complicated for the computations we need to perform. For these maps it is worth spending some time in finding a stabilisation, which are unfoldings involving only one parameter. A germ $\mathcal F=(F_y,y)$ is a \emph{stabilisation} if $F_{y_0}$ is stable for $y_0\in \C\setminus \{0\}$.
\end{re}
 We do not know a method to produce stabilizations other than just trial and error, but  a candidate can be checked to be a stabilization in the following way:

Let $J_y$ be the \emph{relative jacobian ideal} of $I^k(\mathcal F)$, i.e.  the ideal generated by the divided differences of an unfolding $\mathcal F(z,y)=(F_y(z),y)$. To be precise, $J_y$ is generated by the maximal minors of the matrix of partial derivatives, only with respect to $z$, of the generators of $I^k(\mathcal F)$. Inspection of the divided differences gives the equality
\[D^k(\mathcal F)\cap\{y=y_0\}=D^k(F_{y_0}).\]
By Theorem \ref{MM}, the germ  $F_{y_0}$ is stable for all $y_0\neq 0$  if and only if \[D^k(\mathcal F)\cap V(J_y)\subseteq \{y=0\}.\] This can be checked with the help of \textsc{Singular}, as follows.

\begin{propo}
With the previous notations, $\mathcal F$ is a stabilization of $F_0$ if and only if $y\in \sqrt{J_y+I^k(\mathcal F)}$.
\end{propo}

This method has been used to find stabilizations $(\C^6,0)\to(\C^7,0)$ of $Q_2$, $Q_3$ and $Q_4$ mapping $(z,y,x,t,s,u)$, respectively, to
\[\big(z^5+u^2z^3+tz^2+(s^2-x)z,(y+z)^5+sz^3+(u^4+x)z,y,x,t,s,u\big),\]
\[\big(z^5+(u^2+x)z^3-sz^2+(x^2+t)z,(y+z)^6-tz^2+(u^5+sx)z,y,x,t,s,u\big),\]
\[\big(z^8+syz^3+u^6z^2-xz,(y+z)^6+u^2z^4+sz^2+tz,y,x,t,s,u\big).\]

This covers the required techniques to check $\mathcal A$-finiteness and find stabilization and stable unfoldings. Because of its topological nature, computing $\mu_I$ directly is a hard task; we do it via Mond's conjecture \cite{Mond:1991}.

\begin{conj}\label{MondConjecture}Let $F\colon (\C^n,0)\to (\C^{n+1},0)$ be an $\mathcal A$-finite map-germ in Mather's nice dimensions \cite{MatherVI}. Then
\[\mu_I(F)\geq \mathcal A_e\text{-codim}(F),\]
with equality in the weighted-homogeneous case. 
\end{conj}

\begin{re} Mond's conjecture is known to be true in some cases. As explained in Section \ref{secOhmotoFormula}, it holds for $n\leq 2$, but also for fold map-germs, by work of Houston \cite{Houston:1998}, and for singularities of corank 1  with $\mathcal{A}_e$-codimension 1, by  work of Cooper, Mond and Atique \cite{Cooper:2002}.
\end{re}

Our strategy to compute $\mu_I$ for a weighted-homogeneous germ is based on results from \cite{Fernandez-de-Bobadilla:2016} and consists on computing $\mathcal{A}_e\text{-codim}(F)$ first  and then justifying that Mond's conjecture holds for $F$. 
The $\mathcal A_e$-codimension of weighted-homogeneous $\mathcal A$-finite map-germs can be computed with \textsc{Singular}, as follows: let $F\colon (\C^n,0)\to (\C^{n+1},0)$ be an $\mathcal A$-finite map-germ. Let $g\in \mathcal O_{n+1}$ be a function-germ such that $g=0$ is a reduced equation for the image of $F$, and let $J(g)$ be the jacobian ideal of $g$. Then 
\[\mathcal A_e\text{-codim}(F)=\dim_{\mathbb C}\frac{(f^*)^{-1}(J(g)\cdot \mathcal O_n)}{J(g)}.\]

To check that Mond's conjecture holds for $F$, let $\mathcal F\colon(\C^n\times \C^r,0)\to (\C^{n+1}\times \C^r,0)$ be either a stable unfolding or a stabilisation of $F$. Let $G$ be an equation of the image of $F$ and $\mathcal G$ be an equation of the image of $\mathcal F$ which specialises to $G$. Let \[M_y(\mathcal G)=\frac{J(\mathcal G)}{J_y(\mathcal G)},\]
where $J(\mathcal G)$ is the jacobian ideal and $J_y(\mathcal G)$ the relative jacobian ideal of $\mathcal G$. 

\begin{teo}[\cite{Fernandez-de-Bobadilla:2016}, Theorem 6.1]
Let $F$ and $M_y(\mathcal G)$ be as above. If $M_y(\mathcal G)$ is a Cohen-Macaulay module, then $F$ satisfies Mond's conjecture.
\end{teo}

We have used this criterion on stable unfoldings of our new samples, with the exception of $Q_2$, $Q_3$ and $Q_4$ where computations became unfeasible. Mond's conjecture for these three examples was checked by means of the stabilisations above, instead of stable unfoldings.

\begin{re}\label{NoFirstExamples}
As pointed out before, the new singularities used for interpolation in the cases $n=4$ and $5$ were not our first candidates. Observe that not all choices of $(w,d)$ have $\mathcal A$-finite map-germs associated to them (for example, all $(w,d)$ for which the $\mu_I$ formula predicts a non-integer value). The first $\mathcal{A}$-finite germs we found had extremely high $\mathcal A_e$-codimension, making  impracticable to check Mond's conjecture for them. Our strategy was 1) assume that Mond's conjecture holds for these maps. 2) Use their conjectured values of $\mu_I$ to obtain a candidate $\mu_I$ formula. 3) Use a computer to find weights and degrees $(w,d)$ with small (conjectured) $\mu_I$ values and such that they determine enough linearly independent equations. 4) Try to find $\mathcal A$-finite candidates for these $(w,d)$ and check that they satisfy Mond's conjecture. 5) Reproof the $\mu_I$ formula by sampling these new examples.
\end{re}

\bibliography{MyBibliography} 

\begin{thebibliography}{10}

\bibitem{Altintas:2014}
A.~Alt{\i}nta\c{s}~Sharland.
\newblock Examples of finitely determined map-germs of corank {$2$} from
  $n$-space to {$(n+1)$}-space.
\newblock {\em Internat. J. Math}, page~17, 2014.

\bibitem{Cooper:2002}
T.~Cooper, D.~Mond, and R.~Wik Atique.
\newblock Vanishing topology of codimension 1 multi-germs over {$\mathbb R$}
  and {$\mathbb C$}.
\newblock {\em Compositio Mathematica}, 131(2):121--160, 2002.

\bibitem{Jong:1991}
T.~de~Jong and D~van Straten.
\newblock {\em Disentanglements}, volume 1462 of {\em Lecture Notes in Math.
  Singularity theory and its applications, Part I (Coventry, 1988/1989).}
\newblock Springer, Berlin, 1991.

\bibitem{Fernandez-de-Bobadilla:2016}
J.~Fern\'andez~de Bobadilla, J.J. Nu{\~n}o~Ballesteros, and
  G.~Pe{\~n}afort~Sanchis.
\newblock A {J}acobian module for disentanglements and applications to {M}ond's
  conjecture.
\newblock {\em Revista Matem{\'a}tica Complutense}, 32:395--418, 2019.

\bibitem{Houston:1998}
K.~Houston.
\newblock On singularities of folding maps and augmentations.
\newblock {\em Math. Scand.}, 82:191--206, 1998.

\bibitem{Houston:1999}
K.~Houston and N.~Kirk.
\newblock On the classification and geometry of corank 1 map-germs from
  three-space to four-space.
\newblock {\em London Math. Soc. Lecture Note Ser.}, pages 325--351, 1999.

\bibitem{Kazarian2003Multising}
M.~E. Kazarian.
\newblock Multisingularities, cobordisms, and enumerative geometry.
\newblock {\em Russian Mathematical Surveys}, 58(4):665, 2003.

\bibitem{Kazarian:2006}
M.E. Kazarian.
\newblock Thom polynomials.
\newblock In {\em Proc. sympo. ``Singularity Theory and its application''
  (Sapporo 2003)}, volume~43, pages 85--136, 2006.

\bibitem{MacPherson1974}
R.~MacPherson.
\newblock Chern classes for singular algebraic varieties.
\newblock {\em Annals of Mathematics}, 100:421--432, 1974.

\bibitem{MararMondCorank1}
W.~L. Marar and D.~Mond.
\newblock Multiple point schemes for corank $1$ maps.
\newblock {\em J. London Math. Soc.}, 2(3):553--567, 1989.

\bibitem{MatherIV}
J.N. Mather.
\newblock Stability of $c^\infty$ mappings {IV}: {C}lassification of stable
  germs by $\mathbb{R}$-algebras.
\newblock {\em Publications Math{\'e}matiques de l'I.H.{\'E}.S.},
  (37):223--248, 1969.

\bibitem{MatherVI}
J.N. Mather.
\newblock Stability of $c^\infty$ mappings {VI}: {T}he nice dimensions.
\newblock {\em Lecture Notes in Math. Springer, Berlin}, 192:207--253, 1971.

\bibitem{Milnor1968Singular-points}
J.~Milnor.
\newblock {\em Singular points of complex hypersurfaces}.
\newblock Princeton University Press, 1968.

\bibitem{Mond:1991quasihomogeneous}
D.~Mond.
\newblock The number of vanishing cycles for a quasi homogeneous mapping from
  $\mathbb{C}^2$ to $\mathbb{C}^3$.
\newblock {\em The Quarterly Journal of Mathematics}, 42(1):335--345, 1991.

\bibitem{Mond:1991}
D.~Mond.
\newblock {\em Vanishing cycles for analytic maps, Singularity Theory and
  Applications (Warwick 1989)}, volume 1462.
\newblock Springer, New York, 1991.

\bibitem{Mond:1995}
D.~Mond.
\newblock Looking at bent wires -- {$\mathcal A_e$}-codimensions and the
  vanishing topology of parametrised curve singularities,.
\newblock {\em Math. Proc. Cambridge Phil. Soc.}, 117(2):213--222, 1995.

\bibitem{Ohmoto:2014}
T.~Ohmoto.
\newblock Singularities of maps and characteristic classes.
\newblock {\em Adv. Stud. Pure Math}, 2016.

\bibitem{Rimanyi:2001}
R.~Rim{\'a}nyi.
\newblock Thom polynomials, symmetries and incidences of singularities.
\newblock {\em Invent. Math.}, 143:499--521, 2001.

\bibitem{Rimanyi2002Multiple-Point-}
R.~Rim{\'a}nyi.
\newblock Multiple point formulas -- a new point of view.
\newblock {\em Pacific Journal of Mathematics}, 202(2):475--490, 2002.

\bibitem{Sharland2017ExamplesCorank3}
A.~A. Sharland.
\newblock Examples of finitely determined map-germs of corank 3 supporting
  mond's $\mu \geq \tau$-type conjecture.
\newblock {\em Experimental Mathematics}, 2017.

\bibitem{Thom:1956}
Ren{\'e} Thom.
\newblock Les singularit{\'e}s des applications diff{\'e}rentiables.
\newblock {\em Ann. Inst. Fourier}, 6(195556):43--87, 1956.

\bibitem{SingularSofware}
D.~Wolfram, G.~M. Greuel, G.~Pfister, and H.~Sch\"onemann.
\newblock {\sc Singular} {4-0-2} --- {A} computer algebra system for polynomial
  computations.
\newblock \url{http://www.singular.uni-kl.de}, 2015.

\end{thebibliography}
\bibliographystyle{plain}

\end{document}